\documentclass[11pt,psamsfonts]{amsart}

\usepackage[english]{babel}
\usepackage{amsmath,amssymb,graphicx,enumerate,bbm,caption}
\usepackage[utf8]{inputenc}
\usepackage[T1]{fontenc}

\newcommand{\eps}{{\varepsilon}}

\newcommand{\sgn}{\operatorname{sgn}}
\usepackage{hyperref}
\usepackage{dsfont}
\usepackage{amsmath}
\usepackage{amsthm}
\usepackage{amssymb}
\usepackage{amscd}
\usepackage{amsfonts}
\usepackage{amsbsy}
\usepackage{epsfig,afterpage}
\usepackage{psfrag}
\usepackage{overpic}
\usepackage{latexsym}
\usepackage{pstcol}
\usepackage{graphicx}
\usepackage{mathtools}

\newcommand{\R}{\ensuremath{\mathbb{R}}}

\newcommand{\e}{\varepsilon}

\newtheorem {theorem} {Theorem} [section]
\newtheorem {proposition}  {Proposition}[section]

\newtheorem {definition}  {Definition}[section]

\newtheorem {example} {Example}

\begin{document}

\title [Sliding on codimension 2]{Slow--fast systems and sliding on codimension 2 switching manifolds}

\author[da Silva, Nunes]
{Paulo Ricardo da Silva $^*$ and Willian Pereira Nunes $^*$ }

\address{$^*$  Departamento de Matem\'{a}tica --
IBILCE--UNESP, Rua C. Colombo, 2265, CEP 15054--000 S. J. Rio Preto,
S\~ao Paulo, Brazil}

\email{paulo.r.silva@unesp.br}
\email{willianx3@yahoo.com.br}



\date{}
\maketitle

\begin{abstract}In this work we consider piecewise smooth vector fields $X$ defined in $\R^n\setminus \Sigma$, 
	where $\Sigma$ is a self-intersecting switching manifold.
  A double regularization of $X$  is a 2-parameter family  
	of smooth vector fields $X_{\e.\eta}$, $\e,\eta>0,$ satisfying that $X_{\e,\eta}$ converges pointwise to $X$ on $\R^n\setminus\Sigma$, when $\e,\eta\rightarrow 0$.  We define the sliding region on the non regular part of $\Sigma$ as a  limit of invariant manifolds  of  $X_{\e.\eta}$.
	Since the  double regularization provides a slow--fast system, the GSP-theory (geometric singular perturbation theory) is our main tool.
\end{abstract}


\section{Introduction}

One finds in real life and in various branches of science distinguished phenomena whose mathematical models are 
expressed by  piecewise smooth systems and deserve a systematic analysis. However
sometimes the treatment of such objects is far from the usual techniques or 
methodologies found in the smooth universe.   A good reference for a first reading on the subject is \cite{se}.

One of the phenomena that can occur is the existence of sliding regions in the phase space.
In this paper we discuss the definition of sliding when the switching manifold presents self-intersection.
We begin our discussion by presenting the classical definition of sliding  on a regular surface 
and the difficulties to extend to the general case.

Consider two smooth vector fields $X_+,X_-$ defined in $\R^n=\{(x_1,x_2)\in\R\times\R^{n-1}\}.$
A $1$--cross piecewise-smooth vector field  is
\[X=\dfrac{1}{2}\left[(1+\sgn(x_1))X_++(1-\sgn(x_1))X_-)\right].\]
The set of  discontinuity is the codimension 1 manifold
$\Sigma=\{f(x_1,x_2)=x_1=0\}.$ 

The trajectory of $X$ by points on $\Sigma$ depends
on the Lie derivative (\footnote{As usual, we denote  $Xf=\nabla f.X$.}), which is used to classify the points 
as \textit{sewing} or \textit{slide}:
\begin{itemize}
\item[(i)] $\Sigma_{sw} =\{x\in\Sigma:(X_+f.X_-f)(x) >0 \}$ is the sewing region;
\item[(ii)] $\Sigma_{sl} =\{x\in\Sigma:(X_+f.X_-f)(x) <0 \}$ is the sliding region.
\end{itemize}

\begin{figure}[!htb] 
	\begin{overpic}[width=4cm, angle=90]{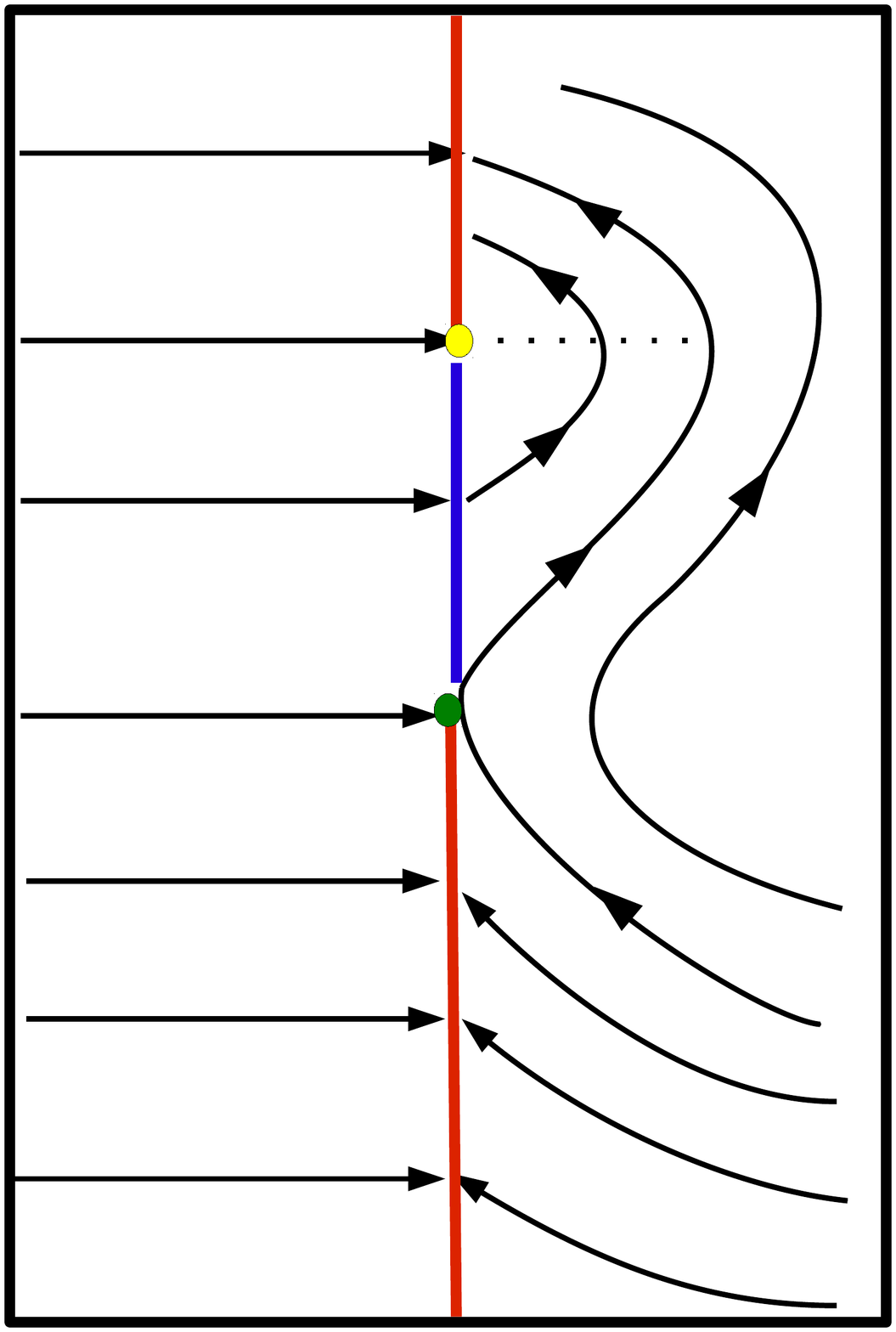}
	\end{overpic}
	\label{buNew}
	\caption{\small Sliding set $\Sigma_{sl}$ (red), sewing set  $\Sigma_{sw}$ (blue), visible fold (green) and invisible fold (yellow).}
\end{figure}

According to Filippov's convention \cite{AF}, for $x \in \Sigma_{sl}$, the trajectory remains in $\Sigma_{sl}$ and obeys
the flow of the sliding vector field $X^S:\Sigma_{sl}\rightarrow\R^n$: 

$$
X^S = (1-\rho)X_+ + \rho X_-, \quad \rho=\frac{X_+f}{X_+f-X_-f}.
$$
The sliding vector field  is a convex combination of $X_+$ and $X_-$ that belongs to the tangent bundle $T\Sigma$. 
The Filippov's convention \cite{AF} also provides first order exit conditions: whenever $\rho = 0$,
one may expect to leave $\Sigma$ to enter in $M_+=f^{-1}(0,+\infty)$ 
with vector field $X_+$, and whenever $\rho = 1$ one may expected to enter in $M_-=f^{-1}(-\infty,0)$ with vector field $X_-$. 
An example of first order exit point is a fold.

Our main tool for  studying the sliding flow  is the \textit{geometric singular perturbation theory} (\textit{GSP-theory}).
The connection between these subjects appears when we regularize the discontinuous vector field. 
A regularization  is  a family of smooth vector fields
$X_{\e}$, with  $\e\geq 0,$ satisfying that $X_{\e}$ converges uniformly to
$X$ in each compact subset of $\R^n  \backslash  \Sigma$ when $\e\rightarrow 0$. 

The Sotomayor-Teixeira regularization \cite{ST} (\textit{ST-regularization})  is based on the use of a
\textit{transition function} (\footnote{by definition, this is a $C^\infty$ function such that $\varphi(t) = -1$ 
for $t \le -1$, $\varphi(t) = 1$ for $t \ge 1$ and $\varphi'(t) >0$ for $-1 < t < 1$.})  $\varphi:\R\rightarrow\R$.
It is the $1$-parameter family $X_{\e}$ given by
\begin{equation}
X_{\e}(x)=\left(\frac{1}{2}+\frac{\varphi_\varepsilon(x_1) }{2} \right)X_+(x)+\left(\frac{1}{2}-\frac{\varphi_\varepsilon(x_1)}{2} \right)X_-(x), \label{ST}
\end{equation}
where $\varphi_{\varepsilon}(x_1)=\varphi(x_1/\varepsilon )$, for $\varepsilon >0$. 

Considering a blow-up $x_1\rightarrow x_1\e$, the trajectories of (\ref{ST}) are the solutions of  a slow-fast system
\begin{equation}
\varepsilon\dot{x}_1=\alpha(x_1, x_2, \varepsilon), \quad \dot{x_2}=\beta(x_1, x_2, \varepsilon).\label{SP}
\end{equation}
We 
can apply the GSP-theory to get information about its phase portrait for $\e\sim 0$ (see for 
instance \cite{BPT,  LST,LST2,LST3,LST4}).
On the paper \cite{LST2} one has the proof  that the reduced dynamics on the critical manifold $\mathcal{S}_0=\{\alpha (x_1,x_2,0)=0\}$ is equivalent
to the dynamics of sliding vector field on $\Sigma_{sl}$. Consequently $\Sigma_{sl}$  is the limit of invariant manifolds $\mathcal{S}_{\e}$ of 
$X_\e$. We can roughly say that the Filippov's approach and the singular perturbation approach provide
the same description of the sliding vector field on $\Sigma.$
However, when the discontinuity occurs in a set $\Sigma\subseteq M$ with 
codimension greater than $1$, the sliding region can not be defined via Filippov's convention. 
The main goal of this paper 
is to study slide on $\Sigma$ with codimension $2$. 
We refer \cite{BCS, Diec1, Diec2, Diec3, NB} for related problems and for  
an introduction to the subject. \\

\subsection{Set of discontinuity  $\Sigma\subseteq \R^n$ with codimension 2 points.}

Consider now four smooth vector fields  $X_{++},X_{-+},X_{--},X_{+-}$ defined in $\R^n=\{(x_1,x_2,x_3)\in\R\times\R\times\R^{n-2}\}$
and the  $2$--cross piecewise-smooth vector field   
\begin{equation}X=\dfrac{1}{4}\left[aX_{++}+bX_{+-}+cX_{-+}+dX_{--}\right]\label{2cruz}\end{equation} where
\[a=(1+\sgn(x_1))(1+\sgn(x_2)),\quad b=(1+\sgn(x_1))(1-\sgn(x_2)),\]
\[c=(1-\sgn(x_1))(1+\sgn(x_2)),\quad d=(1-\sgn(x_1))(1-\sgn(x_2)).\]

Denote $\Sigma_1=\{x_1=0\}\subset \R^n$ and  $\Sigma_2=\{x_2=0\}\subset \R^n$. The set
$\Sigma=\Sigma_1\cup\Sigma_2$ is  the switching manifold and the phase space is divided into four regions, denoted by

\begin{eqnarray*}
M_{++}:  \  x_1>0, x_2>0, &M_{+-}:  \ x_1>0, x_2<0,\\
M_{--}:  \  x_1<0, x_2<0, &M_{-+}:  \ x_1<0, x_2>0.
\end{eqnarray*}

We also use the following notation
\begin{equation*}
\Sigma_1^{\pm}= \left\{x_1=0, x_2 \gtrless 0 \right\}, 
\Sigma_2^{\pm}= \left\{x_2=0, x_1 \gtrless 0 \right\}.
\end{equation*}

The codimension 2 switching set is $\Sigma_{00}= \left\{x_1=x_2=0\right\}.$
Inspired in the  regular case, we try to find a sliding vector field 
as a convex combination of $X_{++},X_{-+},X_{--}$ and $X_{+-}$:
\begin{equation}
X \in \left\{ \sum_{s \in \{-,+\}^2} \lambda_s X_s, \quad \sum_{s \in \{-,+\}^2}\lambda_s=1 \right\}, \quad
X.x_i=0, i=1,2. \label{Filippov}
\end{equation}
Clearly it  is an indeterminate system. Thus the requirement of $X$ being tangent to $\Sigma_1$ and $\Sigma_2$ is not  
sufficient to  characterize a convex combination of $X_{++},X_{-+},X_{--}$ and $X_{+-}$.

In  \cite{PS} we propose  a new definition for sliding  and sewing, linked to the regularization considered. 
Our definition coincides with the definition of Filippov in the regular case. Thus, we can say that our definition generalizes the previous one.

Let $\varphi: \R \longrightarrow \R$ be a transition function. The double regularization is
\[
X_{\varepsilon, \eta}(\textbf{x})=\frac{1}{4}\left(\sum_{s} 
\left(1 +s_1 \varphi\left(\frac{x_1}{\varepsilon} \right)\right)\left(1 +s_2 
\varphi\left(\frac{x_2}{\eta} \right)\right)X_s(\textbf{x})  \right) \label{dreg}
\] where $s=(s_1,s_2)\in\{-,+\}^2$.

\begin{definition} (Slide depending on regularization $X_{\e,\eta}$.)  We say that $p \in \Sigma_{00}$ is a sliding point of \eqref{2cruz}
	if there exist an open neighborhood $U \subset \R^n$ of $p$ and a family of smooth manifolds
	$S_{\e,\eta} \subset U$ defined for all $\e,\eta>0$ such that: 
	\begin{itemize}
		\item For each $\e,\eta$, $S_{\e,\eta}$ is invariant by the restriction of $X_{\e,\eta}$ to $U$.
		\item For each compact subset $K \subset U$, the sequence $S_{\eps,\eta} \cap K$ converges to 
		$\Sigma_{00} \cap K$ when $(\e,\eta)\rightarrow (0,0)$  in some given Hausdorff metric $d_H$ on compact sets of 
		$\R^n$. 
	\end{itemize}
\end{definition}

We remark that the sliding condition is open, that is, if $p\in \Sigma_{00}$  is a sliding  point then there exists an open neighborhood
 $I\subset \Sigma_{00}$ such that any $q\in I$ is a sliding point.\\

Briefly, we list below the  results we have proved in this article. When considering $(\e, \eta) \rightarrow (0,0)$ we can obtain different limit situations depending on the iteration between the parameters. We call  \textit{regularization curve} a path
$\psi (\e,\eta)=0,$ with $\psi (0,0)=0$, in the parameter space $(\e, \eta)$.
\begin{itemize}
\item (Regularization curve $\eta=k\e$) If $\Sigma_{00}$ is the codimension 2 switching manifold then
the sliding region in  $ \Sigma_{00}$ (linked to $X_{\e,k\e}$)
is characterized by the signal of a smooth function.
See \textit{Theorem \ref{teo1}.}

\item (General regularization curve) If $\Sigma_{00}$ is the codimension 2 switching manifold and  the singular points of  $X_s$, $s \in \{-,+\}^2$, 
are not in $\Sigma$ then the sliding region in  $ \Sigma_{00}$  (linked to $X_{\e,\eta}$)
depends on the interactions between the  parameters $\varepsilon$ and $\eta$. See 
\textit{Theorems} \ref{teo_sliding} and \ref{teo_eqsliding}.\\
\end{itemize}

\begin{figure}[!htb] 
	\begin{overpic}[width=10cm]{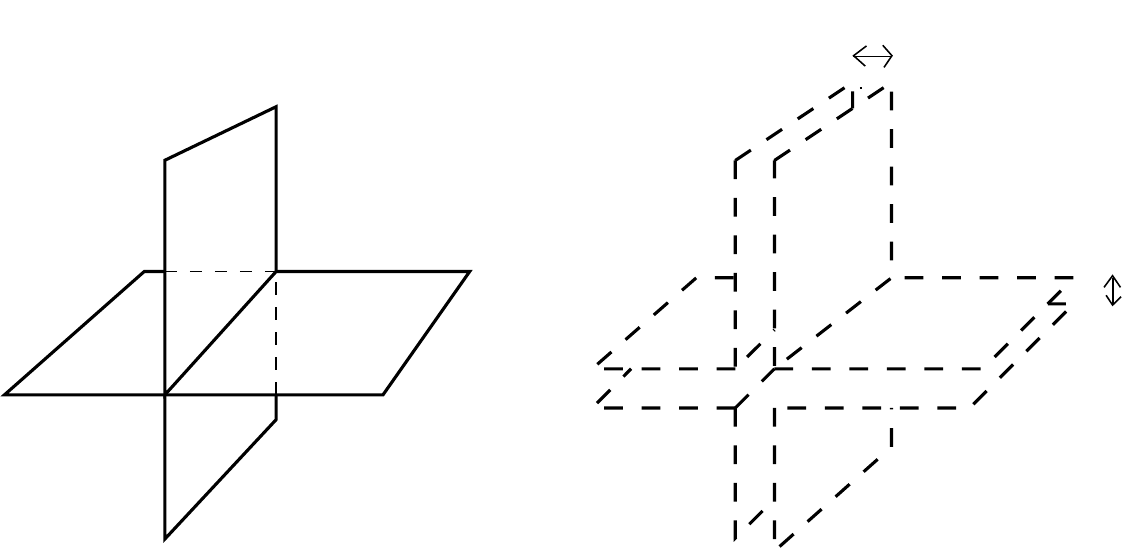}
			\put(75,48){$2\e$}\put(100,22){$2\eta$}
	\end{overpic}
	\label{buNew}
	\caption{\small A codimension 2 switching manifold $\Sigma$ and double regularization.}
\end{figure}

The paper is organized as follows. In Section \ref{s2} we give preliminary definitions and remember 
the main results of GSP-theory.
In Section \ref{s3} we consider a regularizing curve of the kind $\eta= k\varepsilon$, $k>0$. 
Combining  blowing-up technique and 
Fenichel's theory we give sufficient conditions for identifying the sliding region. In Section \ref{s4} 
we state and prove the main result, which consists in conditions for identifying the sliding using the 
 parameters of the double regularization. In Section \ref{s5}, we study a class of planar 
quadratic system, that is useful to determine the sliding regions. In Section \ref{s6} some examples 
are presented to illustrate the main results.

\section{Basic Tools}\label{s2}

This section is dedicated to presenting preliminary results that will be necessary to prove our main results.

\subsection{Singular Perturbation Tools}
Let $\mathcal{U}\subseteq\R^{n+m}$ be an open subset. A singular perturbation problem in $\mathcal{U}$ 
is a differential system which can be written 
like
\begin{eqnarray}
\e\dot{x}_1= \alpha(x_1,x_2,\varepsilon),\quad  
 \dot{x}_2= \beta(x_1,x_2,\varepsilon),\label{padrao1}
\end{eqnarray}
with $\alpha,\beta$ smooth functions, $\varepsilon \in (-\varepsilon_0, \varepsilon_0)$, $\varepsilon_0 >0$ small and 
$(x_1,x_2) \in \mathcal{U}\subseteq \R^n \times \R^m$.  Equivalently, after the time rescaling 
$\tau=t/\varepsilon $, system \eqref{padrao1} becomes
\begin{eqnarray}
x_1'= \alpha(x_1,x_2,\varepsilon),\quad 
x'_2=\e \beta(x_1,x_2,\varepsilon). \label{padrao2}
\end{eqnarray}

The systems (\ref{padrao1}) and (\ref{padrao2}) are called \emph{slow system} and  \emph{fast system},  
respectively.
 By setting $\varepsilon =0$ in (\ref{padrao1}) and (\ref{padrao2}),
we obtain two different limit problems, the \textit{reduced problem}
\begin{eqnarray}
\dot{x}_2= \beta(x_1,x_2,0),\quad 0= \alpha (x_1,x_2,0),  \label{slow}
\end{eqnarray}
and the \textit{layer problem}
\begin{eqnarray}
x'_1= \alpha (x_1,x_2,0),\quad  x'_2= 0.  \label{fast}
\end{eqnarray}

Under adequate assumptions, $\alpha (x_1,x_2,0)=0$ defines a 
manifold $\mathcal{S}$, that will be called \textit{critical manifold}, on which (\ref{slow}) defines a 
dynamical system. But at the same time 
$\mathcal{S}$ is the set  of equilibrium points of (\ref{fast}). So, appropriately combining results 
on the dynamics of these two limiting problems, we obtain results on the dynamics of the 
singularly perturbed problem, for $\varepsilon$ sufficiently small.

Consider system (\ref{padrao2}) suplemented by $\varepsilon'=0$,
\begin{eqnarray}
x'_1= \alpha(x_1,x_2\varepsilon),\quad
x'_2= \e\beta(x_1,x_2\varepsilon), \quad
\varepsilon'=0,
\label{padrao3}
\end{eqnarray}
which is defined in $\mathcal{U} \times (-\varepsilon_0, \varepsilon_0)$. The vector field associated 
to (\ref{padrao3}) will be denoted by 
\begin{equation}\label{xf}
X(x_1,x_2,\varepsilon)=( \alpha(x_1,x_2,\varepsilon), \e\beta(x_1,x_2,\varepsilon), 0),
\end{equation}
with $(x_1,x_2,\e) \in \mathcal{U} \times (-\varepsilon_0, \varepsilon_0)$. By calculating the eigenvalues of 
$LX(x_1,x_2,0)$, with $(x_1,x_2)\in\mathcal{S}$, we have that $\lambda =0$ is a trivial eigenvalue of algebraic 
multiplicity $m+1$.  The remaining eigenvalues are called nontrivial eigenvalues. We 
denote the numbers of nontrivial eigenvalues in the left half plane, in the 
imaginary axis and in the right half plane by  $k^s$, $k^c$ and $k^u$, respectively.

Let $\mathcal{S}_r\subset\mathcal{S}$ be the open set where the nontrivial eigenvalues are nonzero. 
The manifold $\mathcal{S}_r$ 
can be characterized as
$$\mathcal{S}_r=\{(x_1,x_2) \in \mathcal{S}: \text{rank\ } D_{x_1}\alpha(x_1,x_2,0)=n \}.$$
$\mathcal{S}_r$ can be parametrized, locally, by solving the equation $\alpha(x_1,x_2,0)=0$, using 
the implicit function theorem. Let $\mathcal{S}_h \subset \mathcal{S}_r$ be the open set where all the nontrivial 
eigenvalues have nonzero real part, i.e, compact sets $K \subset \mathcal{S}_h$ are normally hyperbolic invariant manifolds 
of (\ref{fast}).

Next theorem is a classical result due to Fenichel, and its proof  can be founded in \cite{F, S}.

\begin{theorem}\label{Fenichel}
 Let $X(x_1,x_2,\e)$,
$\varepsilon \in (-\varepsilon_0, \varepsilon_0)$ be the  family of smooth vector fields on $\mathcal{U}$ given by \eqref{xf} 
and  $\mathcal{S}$ its slow manifold. Let $\mathcal{N}\subset \mathcal{S}_h$ be a $j$-dimensional compact invariant manifold of the
reduced vector field \eqref{slow}, with a $j+ j^s$-dimensional local stable manifold $W^s$ 
and a $j+ j^u$-dimensional local unstable manifold $W^u$. 
\begin{itemize}
\item[i-)] There exist $\e_1 > 0$ and a  family of smooth manifolds $\mathcal{N}_\varepsilon$ with 
$\varepsilon \in (0, \varepsilon_1)$ such that $\mathcal{N}_0=\mathcal{N}$ and 
$\mathcal{N}_\varepsilon$ is an invariant manifold
of $X(x_1,x_2,\e)$;
\item[ii-)] There are a family of smooth manifolds $(j+j^s+k^s)$-dimensional $\mathcal{N}_\varepsilon^s$
 with $\varepsilon \in (0, \varepsilon_1)$ and  family of smooth manifolds 
$(j+j^u+k^u)$-dimensional $\mathcal{N}_\varepsilon^u$ with $\varepsilon \in (0, \varepsilon_1)$ 
such that  the manifolds $\mathcal{N}_\varepsilon^s$ and $\mathcal{N}_\varepsilon^u$ are, locally, 
stable and unstable manifolds of $\mathcal{N}_\varepsilon$.
\end{itemize}
\end{theorem}

\subsection{Sliding region depending on the regularization}

In this section, we introduce some concepts for piecewise smooth systems.

\begin{definition}\label{defmcross} Let  $0 < m \leq n$ be integers.
\begin{itemize}
\item[a.]  The subset $\Sigma \subset \R^n$ given by
$\Sigma = \left\{\prod_{k=1}^m x_k=0 \right\}$
is called a $m$-cross.
\item[b.] A $m$-cross piecewise-smooth vector field defined on $\R^n$ is a vector field  of the kind 
$$
X=\dfrac{1}{2^m} \sum_{s} L_s(\sgn(x))X_s
$$
where $\{X_s\}$ is a collection of $2^m$ smooth vector fields, $s=(s_1,...,s_m)\in \{-,+\}^m$, 
$\sgn(x)= (\sgn(x_1), \cdots ,\sgn(x_m)) \in \{-,+\}^m$ and $$L_s(y)=\frac{1}{2^m} \prod_{k=1}^m (1+s_ky_k).$$

\end{itemize}
\end{definition}

Thus a $1$--cross piecewise-smooth vector field defined on $\R^2$ is 
\[X=\dfrac{1}{2}\left[(1+\sgn(x_1))X_++(1-\sgn(x_1))X_-)\right]\] 
and a $2$--cross piecewise-smooth vector field defined on $\R^3$ is
\[X=\dfrac{1}{4}\left[aX_{++}+bX_{+-}+cX_{-+}+dX_{--}\right]\] where
\[a=(1+\sgn(x_1))(1+\sgn(x_2)),\quad b=(1+\sgn(x_1))(1-\sgn(x_2)),\]
\[c=(1-\sgn(x_1))(1+\sgn(x_2)),\quad d=(1-\sgn(x_1))(1-\sgn(x_2)).\]

\begin{figure}[!htb] 
\centering
\includegraphics[scale=0.25]{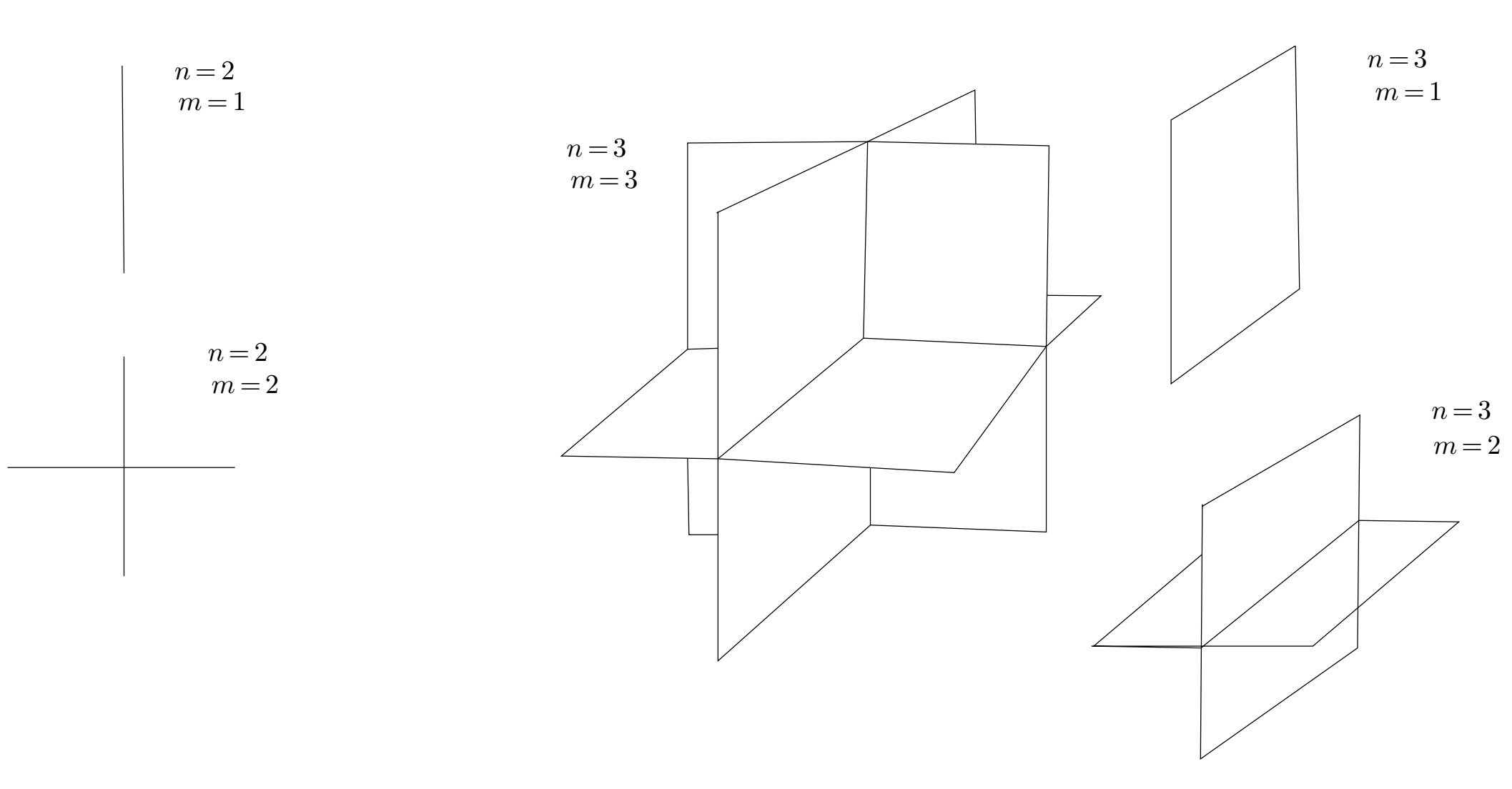}
\caption{\small Possible $m$-cross on $\R^2$ and on  $\R^3$.}
\end{figure}

Let  $\varphi: \R \longrightarrow \R$ be a  monotonic transition function. We are going to use a particular kind of 
regularization, which is induced by $\varphi$.

\begin{figure}[!htb] 
\centering
\includegraphics[scale=0.2]{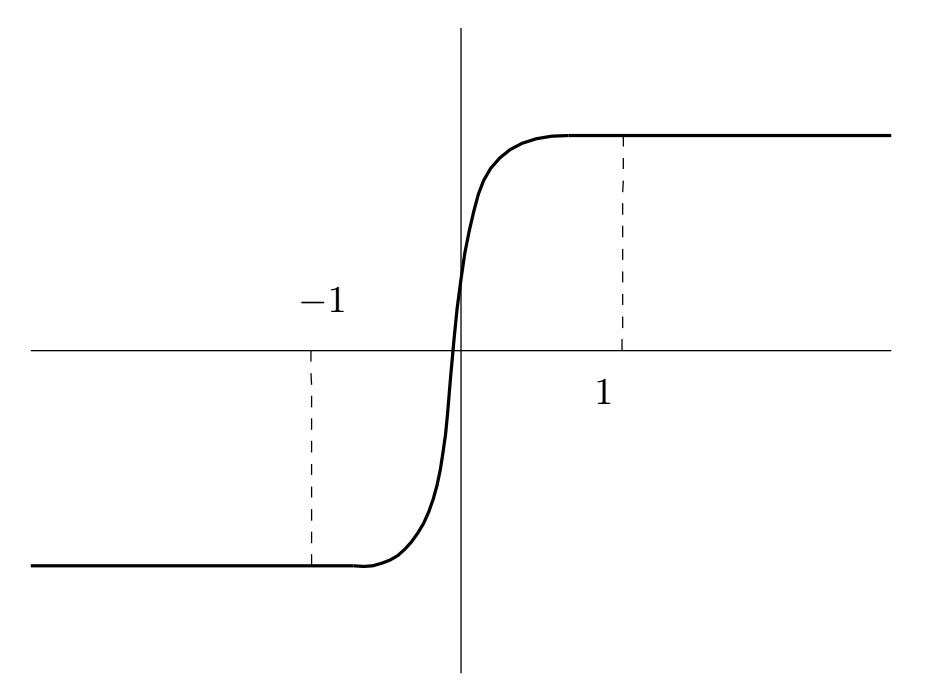}
\caption{\small Sketch of a transition function.}
\end{figure}

\begin{definition}
Given a transition function $\varphi$, the regularization of $m$-cross piecewise-smooth vector field $X$, 
is the family of smooth vector fields

\begin{equation}
X_\varepsilon^\varphi=\frac{1}{2^m}\sum_{s} L_s\left( \varphi\left(\frac{x_1}{\varepsilon_1}\right), \cdots,\varphi\left(\frac{x_m}{\varepsilon_m}\right)        \right)X_s,
\end{equation}
with  $s \in \{-,+\}^m$ and $\varepsilon \in (\R_{+})^m$.
\end{definition}

Note that induced regularization of a $1$-cross is the well known Sotomayor-Teixeira regularization.

Denote 
$$
\Sigma_s=\{\sgn_0(x_1)=s_1,\cdots, \sgn_0(x_n)=s_n  \},
$$
where $\sgn_0(x)$ is the sign function extended to $0$ by $\sgn_0(0)=0$.
Consider the  stratification of $\R^n= \bigcup_s \Sigma_s,$
where the union is taken over all sign vectors  $s=(s_1,\cdots,s_m) \in \{0,-,+\}^m$. 
Each $\Sigma_s$ is a submanifold of codimension equal to the number $z(s)$ of zeros in the sign vector $s$. 
Notice that this induces a
stratification of the $m$-cross
$$
\Sigma = \bigcup_{z(s)>0} \Sigma_s,
$$
where the union is taken over all sign vectors $s$ such that  $z(s)>0$.  \\

A \textit{regularizing curve} is a continuous parametrized curve $\e(\mu)$ in the parameter space 
$(\R_{+})^m$ such as $\displaystyle\lim_{\mu \rightarrow 0} \varepsilon(\mu)=0$.\\

\begin{definition}\label{deslize}
Let $\Sigma_s$ be one of the strata of $\Sigma$. We say that a regularizing curve $\e(\mu)$ produces
a sliding along $\Sigma_s$ if there exists a smooth manifold $S$ of codimension $z(s)$ in product space 
$(\mu,x) \in \R_{\geq 0} \times \R^n$ such that:
\begin{itemize}
\item[i.] $S \cap \{ \mu=0\}=\Sigma_s$;
\item[ii.] For each $\mu_0>0$, $S_{\mu_0}=S \cap \{ \mu=\mu_0\}$ is an invariant manifold for the induced vector field 
$X_{\varepsilon(\mu_0)}^\varphi$.
\end{itemize}
\begin{figure}[!htb] 
\centering
\includegraphics[scale=0.3]{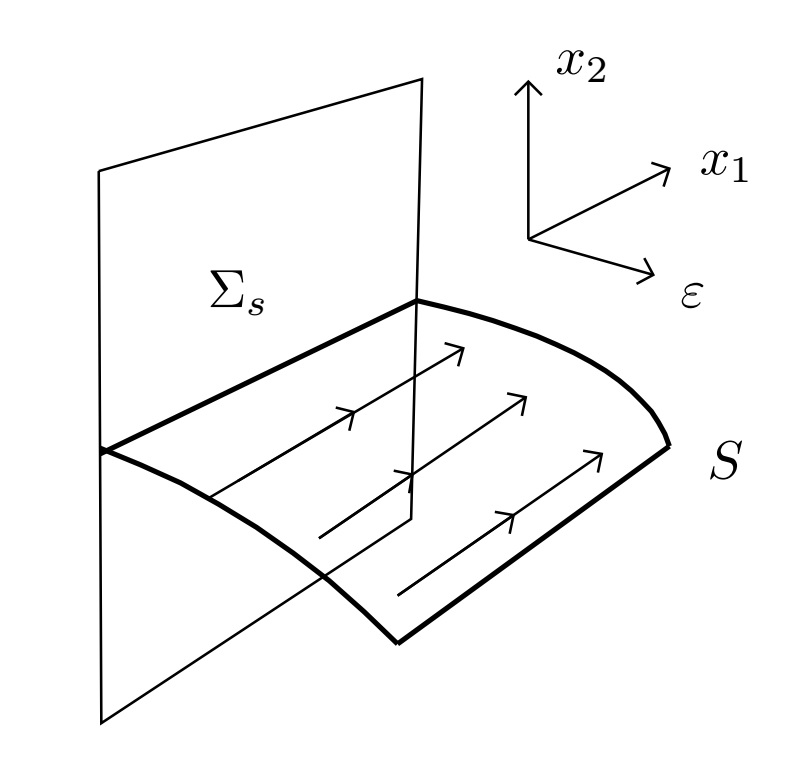}
\caption{\small Sliding with $n=2$, $m=1$.}
\end{figure}
\end{definition}

\section{ $2$-cross and regularization curve $(\e,k\e)$} \label{s3}

In this section we study the sliding associated with $2$-cross piecewise-smooth vector field $X$ on $\R^n$ with $n>2$.
Without loss of generality we assume that $n=3$ and 
\begin{equation}
X=\frac{1}{4} \sum_{s \in \{-,+\}^2} L_s( \sgn(\textbf{x}))X_s \label{dsist}
\end{equation}
where $X_s(\textbf{x})=(f_s(\textbf{x}),g_s(\textbf{x}),h_s(\textbf{x}))$ is smooth and $\textbf{x}=(x_1,x_2,x_3)$.

Let $\varphi: \R \longrightarrow \R$ be a transition function. Consider the regularization
\begin{equation}
X_{\varepsilon, \eta}(\textbf{x})=\frac{1}{4}\left(\sum_{s \in \{-,+\}^2} 
\left(1 +s_1 \varphi\left(\frac{x_1}{\varepsilon} \right)\right)\left(1 +s_2 
\varphi\left(\frac{x_2}{\eta} \right)\right)X_s(\textbf{x})  \right) \label{dreg}
\end{equation}
and its corresponding differential system
\begin{eqnarray}
\begin{split}
x_1'&= \frac{1}{4}\left(\sum_{s \in \{-,+\}^2} \left(1 +s_1 \varphi
\left(\frac{x_1}{\varepsilon} \right)\right)\left(1 +s_2 \varphi
\left(\frac{x_2}{\eta} \right)\right)f_s(\textbf{x})  \right),\\ 
x_2'&= \frac{1}{4}\left(\sum_{s \in \{-,+\}^2} \left(1 +s_1 \varphi
\left(\frac{x_1}{\varepsilon} \right)\right)\left(1 +s_2 \varphi
\left(\frac{x_2}{\eta} \right)\right)g_s(\textbf{x})  \right), \\
x_3'&=\frac{1}{4}\left(\sum_{s \in \{-,+\}^2} \left(1 +s_1 \varphi
\left(\frac{x_1}{\varepsilon} \right)\right)\left(1 +s_2 \varphi
\left(\frac{x_2}{\eta} \right)\right)h_s(\textbf{x})  \right). \label{sist1}
\end{split}
\end{eqnarray}

The sliding region on $\Sigma_s$, for $s=(s_1,s_2) \neq (0,0)$, has been extensively studied and agree with the 
Filippov's convention.  However, on $\Sigma_{00}$ the classical approach does not apply, as 
seen in Section 1. Thus, we  adopt the definition of sliding region  introduced in the previous section.\\

First of all, we consider the double regularization with only one parameter $\e$, that is, we choose a regularization curve of the kind
$(\e , K\e)$, with $K>0$. Thus, in system (\ref{sist1}) we replace $\eta$ by $K\e$.

\begin{theorem}\label{teo1}
	Consider the $2$-cross piecewise smooth vector field \eqref{dsist}. There exists
	a smooth function $D(x_3)$ satisfying that if $D(x_3)\neq 0$ then $(0,0,x_3)\in\Sigma_{00}$
	is a sliding point according to regularization \eqref{dreg}, with $\eta=K\e$.
\end{theorem}
\begin{proof}
Consider the regularized system \eqref{sist1}. The blow-up 
$x_1\rightarrow\e x_1$ and $x_2\rightarrow K\e x_2$, provides 
\begin{equation}
{x}_1'=  X_1({\textbf{x}},\e),\quad {x}_2'= X_2({\textbf{x}},\e) \quad x_3'=X_3({\textbf{x}},\e), \label{slow_sist1}\end{equation}
where 
\begin{eqnarray*}
\begin{split}
 X_1({\textbf{x}}, \e)&=  \frac{1}{4\e}\left(\sum_{s \in \{-,+\}^2} 
 \left(1 +s_1 \varphi\left({x}_1 \right)\right)\left(1 +s_2 \varphi
 \left({x}_2 \right)\right)f_s(\e{x}_1, K\e{x}_2,x_3)  \right),\\ 
X_2({\textbf{x}}, \e)&= \frac{1}{4K\e}\left(\sum_{s \in \{-,+\}^2} 
\left(1 +s_1 \varphi\left({x}_1 \right)\right)\left(1 +s_2 \varphi
\left({x}_2 \right)\right)g_s(\e{x}_1, K\e{x}_2,x_3)  \right), \\
X_3({\textbf{x}}, \e)&=\frac{1}{4}\left(\sum_{s \in \{-,+\}^2} 
\left(1 +s_1 \varphi\left({x}_1 \right)\right)\left(1 +s_2 \varphi
\left({x}_2 \right)\right)h_s(\e{x}_1, K\e{x}_2,x_3) \right), 
\end{split}
\end{eqnarray*}
with ${\textbf{x}}=({x}_1, {x}_2,x_3)$. System \eqref{slow_sist1} is the slow system and the corresponding vector field is denoted by 
$X_\varepsilon ({\textbf{x}},\varepsilon)= (X_1({\textbf{x}},\e),X_2({\textbf{x}},\e),X_3({\textbf{x}},\e))$.\\

For  ${\textbf{x}}_0 \in S=\{X_1({\textbf{x}},0)=0\} \cap \{X_2({\textbf{x}},0)=0\}$ denote $D_0$ 
the matrix of \[D_{({x}_1,{x}_2)}\left(X_1,X_2 \right)({\textbf{x}}_0,0).\] 
We will prove that $D(x_3)=\operatorname{tr}(D_0).\det (D_0)\neq 0$
implies $(0,0,x_3)$ to be a sliding point of $X$.

For $\varepsilon =0$, $S$ is  the slow manifold of system \eqref{slow_sist1}. Let $S_r \subset S$ be the open subset of $S$ given by
$S_r=\{{\textbf{x}} \in S: \text{rank\ }  D_{({x}_1,{x}_2)}\left(X_1,X_2 \right)({\textbf{x}},0)=2  \}.$
The matrix $D_0$  is 
given by
\begin{equation}
\left[\begin{array}{rrrr}
\displaystyle\frac{\partial X_1({\textbf{x}}_0,0)}{\partial {x}_1} & \displaystyle\frac{\partial X_1({\textbf{x}}_0,0)}{\partial {x}_2} \\ \\
\displaystyle\frac{\partial X_2({\textbf{x}}_0,0)}{\partial {x}_1} & \displaystyle\frac{\partial X_2({\textbf{x}}_0,0)}{\partial {x}_2} 
\end{array}\right]. \label{matrizhip}
\end{equation}
The hypothesis  implies that 
the rank of $D_0$ is $2$ and the eigenvalues of $D_0$ have nonzero real part. In other words, 
${\textbf{x}}_0 \in S_h \subset S_r$. Since $S_h$ 
is a open set  there exists a neighborhood $V\ni{\textbf{x}}_0$ 
 such that $\bar{V} \subset S_h$. Using the Theorem \ref{Fenichel}, there exists a $C^{r-1}$ family of 
manifolds ${V}_\varepsilon$ with $\varepsilon \in (0, \varepsilon_1)$ 
such that ${V}_0={V}$ and ${V}_\varepsilon$ is an invariant 
manifold $X(\textbf{x},\varepsilon)$, proving that $(0,0,x_3)$ is a sliding point for $X$.

\end{proof}

\begin{example}
\end{example}
Let $X$ be a $2$-cross piecewise-smooth vector field defined on $\R^3$, given by 
$X(x,y,z)=X_{\pm  \pm }(x,y,z)$ where
\begin{eqnarray*}
	\begin{split}
		X_{++}=&\left(-1+x^2,-1+y^2,z\right) \ \ &X_{+-}=&\left(-1+xy,1-zy,z\right)\\ 
		X_{--}=&\left(1,1+z^2,z\right) \ \ &X_{-+}=&\left(1+x+z,-1,z\right),
	\end{split}
\end{eqnarray*}
Consider the regularization of $X$ and the directional blow-up $x\rightarrow \e x$,
$y\rightarrow \e y$. System \eqref{slow_sist1} is 
\begin{equation}
\e x'= X_1(\textbf{x});\quad
\e y'= X_2(\textbf{x});\quad
z'= X_3(\textbf{x}), 
\end{equation}

$\textbf{x}=(x,y,z)$,where
$$
X_1=\left(x^3-zx-3x+z+y(x^3+x^2-xz+2x+z)+y^2(-x^2-x)\right)/4
$$
$$
X_2=\left(-xz^2+z^2+y(xz^2-xz-z^2-z-4)+y^2(xz+x+z+1)+y^3(x+1)\right)/4
$$
and $X_3(\textbf{x})=z$. We are going to check if $p=(0,0,0) \in \Sigma_{00}$ is a sliding point. The function $D(z)$ is given by
$$
D(z)= \left( -{z}^{2}-2\,z-7 \right)  \left( 2\,{z}^{3}+4\,{z}^{2}+7\,z+12
\right). 
$$
Since $D(p)=-84 \neq 0$, $p$ is a sliding point.

\section{Sliding depending of the regularization curve} \label{s4}

In this section we study the sliding associated with $2$-cross piecewise-smooth vector fields generated by smooth vector fields  
$F_s:\R^3\rightarrow\R^3, s\in\{-,+\}^2$ satisfying that $F_s(0,0,0)\neq0$.
Initially we consider constant vector fields $X_s,s\in\{-,+\}^2$ defined in $\R^3$ and then we extend our 
result to the general case via tubular flow.\

Assume that $X_s=(a_s,b_s,1)$ with $a_s,b_s \in \R$, $s=(s_1,s_2)\in\{-,+\}^2$.  
Given a transition function $\varphi: \R \longrightarrow \R$, consider the regularization
\begin{equation}
X_{\varepsilon, \eta}(\textbf{x})=\frac{1}{4}\left(\sum_{s} \left(1 +s_1 \varphi
\left(\frac{x_1}{\varepsilon} \right)\right)\left(1 +s_2 \varphi\left(\frac{x_2}{\eta} \right)
\right)X_s(\textbf{x})  \right),\label{dreg2}
\end{equation}
with $\textbf{x}=(x_1,x_2,x_3)$. Here $(\e,\eta)$ is an arbitrary regularization curve. First of all take the directional blow-up 
$x_1\rightarrow \e x_1$ and $x_2\rightarrow\eta x_2$.  The corresponding  differential system is
\begin{equation}
  \left\{ \begin{aligned}
      {x}_1'&= \frac{1}{\varepsilon}\sum_{s }\left(1 +s_1 \varphi({x}_1) \right)\left(1 +s_2 \varphi({x}_2) \right)a_{s}, \\
      {x}_2'&=\frac{1}{\eta}\sum_{s }\left(1 +s_1 \varphi({x}_1) \right)\left(1 +s_2 \varphi({x}_2) \right)b_{s} , \\
      x_3'&=1,\quad      \varepsilon'=0,\quad\eta' =0.\\
  \end{aligned} \right. \label{csist1}
\end{equation}

The flow of system \eqref{csist1} depends essentially on the first two rows. We refer to these two lines as
${x}_1{x}_2$--system.

\begin{equation}
  \left\{ \begin{aligned}
      {x}_1'&= \frac{1}{\varepsilon}\sum_{s}\left(1 +s_1 \varphi({x}_1) \right)\left(1 +s_2 \varphi({x}_2) \right)a_{s}, \\
    {x}_2'&=\frac{1}{\eta}\sum_{s}\left(1 +s_1 \varphi({x}_1) \right)\left(1 +s_2 \varphi({x}_2) \right)b_{s} .
  \end{aligned} \right. \label{csist2}
\end{equation}
Analyzing the dynamics of (\ref{csist2}), we obtain conditions to define the sliding of system $X$. 
Near $(0,0)$ we have  $\varphi({x}_1) \approx {x}_1 $ 
and $\varphi({x}_2) \approx {x}_2$. Using this equivalence, with a time reescaling  $t=\varepsilon \tau$, 
the system (\ref{csist2}) is equivalent to
\begin{equation}
  \left\{ \begin{aligned}
      \dot{{x}}_1&= \sum_{s }\left(1 +s_1 {x}_1 \right)\left(1 +s_2 {x}_2 \right)a_s, \\
      \dot{{x}}_2&=\frac{\varepsilon}{\eta} \sum_{s }\left(1 +s_1 {x}_1 \right)\left(1 +s_2 {x}_2 \right)b_s .
  \end{aligned} \right. \label{csist3}
\end{equation} 

If $\lambda_1=(a_{++}-a_{-+}-a_{+-}+a_{--})/4 \neq 0$  and $\lambda_2=(b_{++}-b_{-+}-b_{+-}+b_{--})/4 \neq 0$
then system \eqref{csist3} is quadratic and it can be written as
\begin{equation}
  \left\{ \begin{aligned}
      \dot{{x}}_1&= \lambda_1 ({x}_1-\alpha_1)({x}_2-\beta_1)-\delta_1,\\
      \dot{{x}}_2&=\frac{\varepsilon}{\eta} \lambda_2 ({x}_1-\alpha_2)({x}_2-\beta_2)-\delta_2,
  \end{aligned} \right. \label{csist4}
\end{equation}
where $\alpha_1=-(a_{++}+a_{-+}-a_{+-}-a_{--})/\lambda_1$,
$\beta_1=-(a_{++}-a_{-+}+a_{+-}-a_{--})/\lambda_1$ and $\delta_1=(\lambda_1 \alpha_1 \beta_1 -
\left(a_{++}+a_{-+}+a_{+-}+a_{--}\right))/4$.
For $\alpha_2, \beta_2$ e $\delta_2$, the expressions are  the same but changing $a_i$ by $b_i$.

\begin{proposition}\label{prop1}
If  system \eqref{csist4} has two equilibria, then only one of them
is a saddle point. 
\end{proposition}
\begin{proof}
The linearization of system \eqref{csist4} is
$$
L({x}_1,{x}_2)= \left[ \begin {array}{cc} \lambda_{1}\, \left( {x}_2-\beta_{1} \right) &
\lambda_{1}\, \left( {x}_1-\alpha_{1} \right) \\ \noalign{\medskip}{\displaystyle \frac 
{\varepsilon\,\lambda_{2}\, \left( {x}_2-\beta_{2} \right) }{\eta}}&{\displaystyle \frac {
\varepsilon\,\lambda_{2}\, \left( {x}_1-\alpha_{2} \right) }{\eta}}
\end {array} \right] 
$$
and it has determinant  given by

\[\det({x}_1,{x}_2)=-{\frac {\lambda_{1}\,\e\,\lambda_{2}\, \left( \beta_{1}\,{x}_1-
\beta_{1}\,\alpha_{2}-{x}_1\beta_{2}+\alpha_{1}\,\beta_{2}-\alpha_{1}\,{x}_2+{x}_2
\alpha_{2} \right) }{\eta}}.\]

If system \eqref{csist4} has two distinct equilibra $P_1,P_2$, we have that $\alpha_1 \neq \alpha_2$ and
$\beta_1 \neq \beta_2$. So $\det({x}_1,{x}_2)=0$ defines a straight line passing through the points
$(\alpha_1,\beta_1)$ and $(\alpha_2,\beta_2)$. Note that, the ratio between the parameters does not affect the straight line
determined by $\det ({x}_1,{x}_2)=0$, for any $ \e, \eta$. So, the intersection  $\dot{{x}}_1=0$ and $\dot{{x}}_2=0$, which 
give us the equilibria, occurs on opposite sides of $\det({x}_1,{x}_2)=0$, that is, $\det(P_1)\det(P_2)<0$, 
proving that one of the points is a saddle. 
\end{proof}

Applying a translation ${x}_1-\alpha_1$, ${x}_2-\beta_1$,  and a time reescaling
$\tau_1=\tau/\lambda_1$ on (\ref{csist4}) we get
\begin{equation}
  \left\{ \begin{aligned}
      \dot{{x}}_1&= {x}_1 {x}_2-\delta_1 \\
      \dot{{x}}_2&=\frac{\varepsilon}{\eta } C ({x}_1-{\alpha}_2)({x}_2-{\beta}_2 )-\delta_2,
  \end{aligned} \right. \label{csist5}
\end{equation}
where $C= \lambda_2/\lambda_1$ and possibly different constants $\alpha_1, \alpha_2$.

\begin{proposition}\label{mainprop}
Let $X$ be the vector field (\ref{csist5}) having two equilibria $P,Q$ where $P=({x_1},{x_2})$ 
denotes the non-saddle equilibrium with  ${x_1}\neq0$.
\begin{itemize}
\item[i.] If $\frac{\varepsilon}{\eta} \rightarrow k $, as $\e ,\eta \rightarrow 0$, with $k >0$, 
then the nature of $P$ and $Q$ do not change, for any $\varepsilon, \eta$;
\item[ii.] If $\frac{\e}{\eta}\rightarrow 0$ or $\frac{\e}{\eta}\rightarrow \infty$ as $\e ,\eta \rightarrow 0$,   
then $P$ is asymptotically stable or asymptotically unstable.  Moreover, if $P$ is asymptotically stable 
(unstable) for  $\frac{\e}{\eta}\rightarrow 0$ then it is asymptotically unstable (stable) for $\frac{\e}{\eta}\rightarrow \infty.$
\end{itemize}
\end{proposition}

\begin{proof}

\begin{figure}[t]
\centering
\includegraphics[scale=0.25]{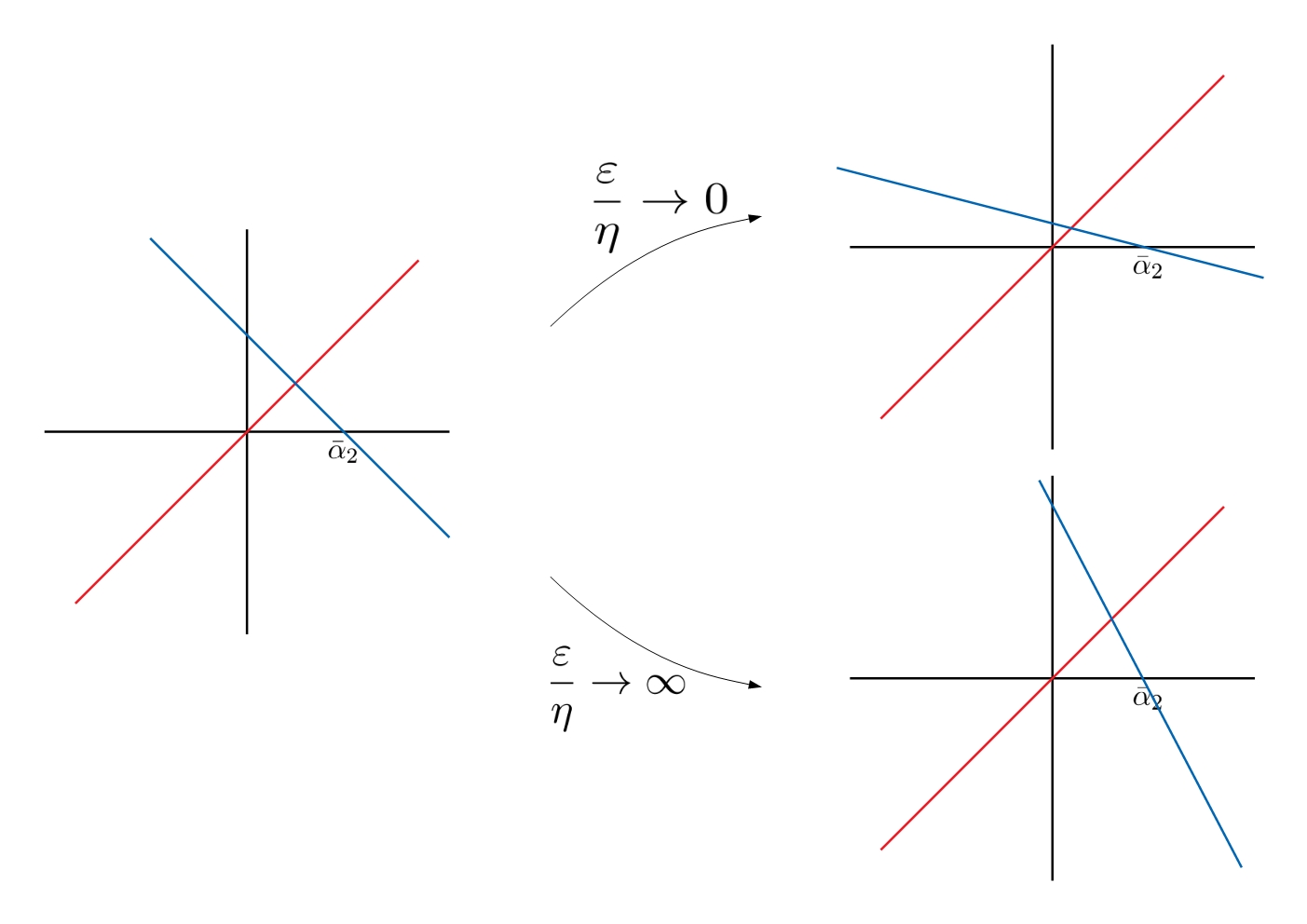}
\caption{\small Lines $r_1$ ($Tr=0$,blue), $r_2$ ($\det=0$, red) and equilibrium points $P,Q$.}
\end{figure}

The linearization of the system (\ref{csist5}) depending of ${x}_1$ and ${x}_2$ is
$$
J= \left[ \begin {array}{cc} {x}_2 & {x}_1\\ \noalign{\medskip}
\displaystyle\frac{\varepsilon C ({x}_2-{\alpha}_2)}{\eta} & \displaystyle\frac{\varepsilon C({x}_1-{\beta}_2)}{\eta} \end {array} \right] 
$$
 Note that the trace and the determinant are given, respectively, by
$$
\text{Tr}({x}_1,{x}_2)={x}_2+{\frac {\varepsilon\,C \left( {x}_1-{\alpha}_2 \right) }{\eta}}
\text{\ and \ } \det({x}_1,{x}_2)=-{\frac {\varepsilon\,C \left( {\beta}_2{x}_1- {\alpha}_2{x}_2 \right) }{\eta}}.
$$
 Tr$({x}_1,{x}_2)=0$ and $\det({x}_1,{x}_2)=0$, define  two straight lines in the plane ${x}_1{x}_2$, 
 which we denote, respectively,
by $r_1,r_2$. The line $r_2$ is independent of the parameters $\varepsilon$ and $\eta$, but 
the inclination of line $r_1$ depends
of the quotient $\varepsilon/\eta$. By Proposition \ref{prop1} 
one of these equillibrium points  is a saddle and the other depends of the line $r_1$. 
Assuming \textbf{i}, the inclination of $r_1$ does not change for $\varepsilon,\eta  \rightarrow 0$, 
that is, the nature of $P,Q$ does not  change.  Now, assume \textbf{ii} with
$\varepsilon/\eta \rightarrow 0 $. In this case $ r_1 $ tends to $y=0$, 
and the stability of $P$  depends on its position in relation to $r_1$ and $r_2$ and  remains
  for $\varepsilon $ and $\eta $ sufficiently small. If $\varepsilon /\eta \rightarrow \infty $ the position
  of $P$ does not change, but  $r_1$ tends to $y={\alpha}_2$, which implies that for sufficiently small values
  of $\varepsilon $ and $\eta $, the stability is  opposite, proving the theorem.
 
\end{proof}

Observe that, the proof of  Proposition \ref{mainprop},\textbf {i}, could be obtained using the Theorem \ref{Fenichel}. 
For only one equilibrium point, the proof is the same.

\begin{theorem}\label{teo_sliding}
Let $X$ be a $2$-cross piecewise-smooth constant vector field defined in $\R^3$ and generated by 
$X_s=(a_s,b_s,1)$ with $a_s,b_s \in \R$, $s=(s_1,s_2)\in\{-,+\}^2$. Consider  the regularization $X_{\e ,\eta}$  
given by \eqref{dreg2} and apply the blow up \eqref{csist1}.
Assume that the equilibria points of the ${x}_1{x}_2$--system \eqref{csist2}  are localized at $(-1,1)^2$. 
\begin{itemize}
\item[i.] If ${x}_1{x}_2$--system has two equilibria, then every point in $\Sigma_{00}$ is  
a sliding point of $X$, for any $\e, \eta$.
\item[ii.] If ${x}_1{x}_2$--system has only one equilibrium and it is asymptotically 
stable or unstable when $\e, \eta\rightarrow 0$, then every point in $\Sigma_{00}$ is a sliding point of $X$.
\end{itemize}
\end{theorem}

\begin{proof}
Let's assume first that we have two equilibria for the ${x}_1{x}_2$--system. According to the Proposition 
\ref{prop1}, one of these points is a saddle for any $\e , \eta$. Thus there are at least two invariant manifolds, 
the stable and unstable manifolds, for the ${x}_1{x}_2$--system. So, every regularizing curve produces a 
sliding along $\Sigma_{00}$, proving the first item.
Now, assume that we have only one equilibrium point and it is asymptotically stable. Assume that the
 regularizing curve is of kind $(\e, K \e)$, $K>0$. According to Proposition \ref{mainprop} item \textbf{i}, 
the equilibrium will remain asymptotically stable when $\e, \eta\rightarrow 0$ . So, we have stable invariant 
manifolds for the ${x}_1{x}_2$--system and it produces a sliding along $\Sigma_{00}$. If we have a 
regularizing curve satisfying $\e/\eta \rightarrow 0$ or $\e/\eta \rightarrow \infty$, when $\e, \eta\rightarrow 0$  
the stability of the equilibrium may change, but for $\e , \eta$ sufficiently small it is 
asymptotically stable or unstable, depending on $X$, which implies that we have a stable (unstable) 
invariant manifold for the ${x}_1{x}_2$--system and it produces a sliding along $\Sigma_{00}$. 
The conclusion is the same, if we had the equilibrium asymptotically unstable.

\end{proof}

%
%
%

\begin{theorem}\label{teo_eqsliding}
		Let $F$ and $X$ be  $2$-cross piecewise-smooth vector fields, generated by smooth vector fields  
		$F_s$  and $X_s$, $s \in \{-,+\}^2$
		 defined on a neighborhood $V \subset \R^3$ of the origin. Assume that $X_s=F_s(0)$ is a non zero constant vector filed
		 and that $F_s.x_i \neq 0$, $i=1,2$ 
		 for all $s \in \{-,+\}^2$. The sliding regions of $F$ and $X$  are the same in a neighborhood of $0$.
\end{theorem}

\begin{proof}
Note that $F$ and $X$ are locally topological equivalent. Since $X_s=F_s(0)$ there exist a neighborhood $0\in U\subset V$ and $\e >0$, 
such that $|| F_s|U - X_s|U || < \e/2$, in the $C^0$-topology. Given a transition function $\varphi$, consider the 
regularizations $F_{\e , \eta}$ and $X_{\e, \eta}$. Note that, $\Vert F_{\e , \eta}(x) - X_{\e , \eta}(x) \Vert$ is expressed by
\begin{eqnarray*}
\begin{split}
\left \Vert\left(\left(\frac{1+\varphi\left(x/\e \right)}{2} \right)(F_+(x)-X_+(x))+\left(\frac{1-\varphi\left(x/\e \right)}{2} \right)(F_{-}(x)-X_{-}(x))\right) \right \Vert\\. 
\end{split}
\end{eqnarray*}
By definition of the transition function we have that,
\begin{equation}
\Vert F_{\e , \eta}(x) - X_{\e , \eta}(x) \Vert \leq \left(\frac{\e}{2}\right)+\left(\frac{\e}{2}\right)=\e.
\end{equation}
Since the vector fields are close, the same occurs with  the invariant manifolds of $X_{\e, \eta}$ and $F_{\e, \eta}$.
 Thus, if the regularization $X_{\e, \eta}$
satisfies the hypothesis of the Theorem \ref{teo_sliding}, $X$ has a sliding region, which implies 
that $F$ has the same  sliding region.
\end{proof}

\section{On the quadratic system (\ref{csist4})}\label{s5}

In this section, we study system  (\ref{csist4}) with parameters $\varepsilon=\eta=1$.
We rewrite it as the following

\begin{equation}
x'= A(x-a)(y-b)-B=F(x,y),\quad
y'= C(x-c)(y-d)-D=G(x,y) .\label{bsist1}
\end{equation}

Our first Theorem establishes classes of affine equivalence of system \eqref{bsist1}, depending on $A,B,C,D$.

\begin{theorem}\label{teoequiv}
System \eqref{bsist1} is affine equivalent, using a rescaling of the independent variable if necessary, to one of the following systems:
\begin{itemize}
	\item [(I)] 	($a\neq c,b\neq d$): \quad $x'= xy-B$,\quad $y'= C(x-1)(y-1)-D$. 
	\item [(II)]    ($a\neq c ,b=d$): \quad $x'= xy-B$, $y'= (x-1)y-D$.
	\item [(III)] ($a= c,b \neq d$): 	\quad $x'= xy-B$, \quad $y'= x(y-1)-D$. 
	\item [(IV)] ($a= c,  b= d,B\neq 0$): \quad $x'= xy-1$, \quad  $y'= xy-D$.
	\item [(V)] ($a= c, b= d,B=0,D\neq 0$): \quad $x'= xy$, \quad	$y'= Cxy-1$.
	\item [(VI)] ($a=c, b=d,B=D=0$): \quad $x'= xy$, \quad $y'= Cxy$.
\end{itemize}

\end{theorem}

\begin{proof}
Consider the change of variables $ x \rightarrow u x + v $, $ y \rightarrow w y+ r $   with $u, v, w $ and $r$
arbitrary constants and get 
\begin{eqnarray}
\begin{split}
x'&= wA\left(x-\left(\frac{a-v}{u}\right)\right)\left(y-\left(\frac{b-r}{w}\right)\right)-B\\ 
y'&= uC\left(x-\left(\frac{c-v}{u}\right)\right)\left(y-\left(\frac{d-r}{w}\right)\right)-D.
\end{split}
\end{eqnarray}
We start the proof considering system (I).  Choose $v = a$ and $r = b$. Considering $t\rightarrow wAt$ system  \eqref{bsist1} becomes
\begin{eqnarray}
\begin{split}
x'={}& xy-\frac{B}{wA}\\ 
y'={}& \frac{uC}{wA} \left(x-\left(\frac{c-a}{u}\right)\right) \left(y-\left(\frac{d-b}{w} \right)\right)-\frac{D}{wA}.
\end{split}
\end{eqnarray}
Since  $a\neq c$ and $b\neq d$ we can take $u = c-a$ and  $w = d-b$. Thus we get
\begin{equation}
x'= xy-\bar{B},\quad
y'= \bar{C}(x-1)(y-1)-\bar{D},
\end{equation}
where $\bar{B}$, $\bar{C}$ and $\bar{D}$ are constant.
The proof for the other systems  is analogous.
\end{proof}

Next propositions is about the dynamics  of systems (I), (II) and (III).

\begin{proposition}
Consider the differential system (II).
\begin{itemize}
\item[i.] For $B\neq D$, the only equilibrium point is  $P=(B/B-D,B-D)$ and for $B=D=0$, all points  
$(x,0)$ with $x\in \R$ are equilibrium points .
\item[ii. ] If $D<0$ and $B=D-\sqrt{-D}$, the equilibrium point is a center.
\end{itemize} 
\end{proposition}
\begin{proof}
The proof of [(i)] follows of a simple computation. For [(ii)], use the change $x\rightarrow x+(B/B-D)$ and $y\rightarrow y+(B-D)$, and obtain 
\begin{equation}
x'= {\frac {x\,D}{\sqrt {-D}}}+x\,y-{\frac {y\,D}{\sqrt {-
D}}}+y,
\quad
y'= {\frac {x\,D}{\sqrt {-D}}}+x\,y-{\frac {y\,D}{\sqrt {-
D}}},\label{eqtipo2}
\end{equation}
with the equilibrium point now at the origin. The linearization at the origin is given by
$$
L= \left[ \begin {array}{cc} -\sqrt {-D}&-{\frac {D-\sqrt {-D}}{\sqrt {-
D}}}\\ \noalign{\medskip}-\sqrt {-D}&\sqrt {-D}\end {array} \right] 
$$
with eigenvalues $\lambda_{1,2}=\pm i\sqrt [4]{-D}$. The matrix formed by the eigenvectors is
$$
M=\left[ \begin {array}{cc} 1/2&1/2\,\sqrt [4]{-D}\\ \noalign{\medskip}0
&1/2\,\sqrt [4]{-D}\end {array} \right]. 
$$
Applying the change of variables on \eqref{eqtipo2} and the time reescaling $t\rightarrow t \sqrt[4]{-D}$, we obtain 
\begin{equation}
\dot{x}=-y,
\quad
\dot{y}=x- \frac{y^2}{2}+\frac{\sqrt[4]{(-D)^3}xy}{2D}.
\end{equation}
The coefficients of $x^2$ and $y^2$ are both equals to zero, so using the classical  Bautin's Theorem the origin is a center point.
\end{proof}

The others aspects of the dynamics of systems (II) are not difficult to study, except this one, 
when the eigenvalue is non hyperbolic. For systems (III), the same problem arise, for 
$B<0$ and $D=B-\sqrt{-B}$ the equilibrium is a center and the proof is similar.

\begin{theorem}
Consider the differential system  (I). If $B=B_1+C^2/(1+C)^2$, $D=D_1+C/(1+C)^2$ and $C\neq 1$ then  it is topologically equivalent to
\begin{equation}
x'= x,
\quad
y'= \beta_1 + \beta_2 x+ x^2 +s x y+ \mathcal{O}(|| \textbf{x}^3 ||).
\end{equation}
\end{theorem}

\begin{proof}
System (I)  with the new parameters is 

\begin{eqnarray}
\begin{split}
x'&= xy-\left(B_1 + \frac{C^2}{(1+C)^2} \right) \\ 
y'&= C(x-1)(y-1)-\left(D_1 + \frac{C}{(1+C)^2} \right). \label{eqtipo1}
\end{split}
\end{eqnarray}
Without loss of generality, assume $C>1$. Note that, at $B_1=D_1=0$ \eqref{eqtipo1} 
has only one equilibrium given by $P=(C/(1+C), C/(1+C))$ and the linearization at $P$
$$
L= \left[ \begin {array}{cc} {\displaystyle \frac {C}{1+C}}&{\displaystyle \frac {C}{1+C}}
\\ \noalign{\medskip}-{\displaystyle \frac {C}{1+C}}&-{\displaystyle \frac {C}{1+C}}\end {array}
 \right] \neq 0
$$
has a double zero eigenvalues. Applying a translation $x-C/(1+C)$, $y-C/(1+C)$ and 
with the time reescaling $\tau_1=\tau (1+C)/C$ on \eqref{eqtipo1} we get:
\begin{eqnarray}
\begin{split}
x'&= -{\frac {-Cx\,y+B_{1}\,C-x\,C-y\,C-x\,y+B_{1}
}{C}}=F(x,y) \\ 
y'&= {\frac {x\,y\,{C}^{2}+Cx\,y-CD_{1}-x\,C-y\,C-D
_{1}}{C}}=G(x,y). \label{tr1}
\end{split}
\end{eqnarray}
Consider the map
$$
(x,y,B_1,D_1) \rightarrow \left(F(x,y),G(x,y) ,Tr(L(x,y)),\det (L(x,y)) \right)
$$
and its linearization has a determinant given by $-C(C-1)$, which is nonzero, so this map is regular for $(x,y,B_1,D_1)=(0,0,0,0)$. 

Consider the following sequence of change of coordinates:
\begin{itemize}
	\item $x\rightarrow x+\dfrac{y}{1+C}$, $y\rightarrow \dfrac{C}{1+C}y-x$;
	\item $y\rightarrow y+\left(\dfrac{B_1C-D_1}{C}\right)$, $x\rightarrow x-\left(\dfrac{\alpha_2}{2\alpha_1}\right)$,
\end{itemize}
where 
$$\alpha_1 =\dfrac{-C^3-2C^2-C}{C^2}  \text{\ and \ } \alpha_2=\dfrac{B_1C^3-C^2D_1-B_1C+D_1}{C^2},$$
we get the system
\begin{eqnarray}
\begin{split}
x'&= y \\ 
y'&=b_{0,0} +b_{0,1} y+ b_{2,0} x^2+b_{1,1} xy +y^2
\end{split}
\end{eqnarray}
where
\begin{itemize}
	\item $b_{0,0}= b_{0,0}(B_1,C,D_1)$;
	\item $b_{0,1}=\dfrac{1}{2}\left(\left( \dfrac{(C+1)^2}{C}\right)B_1 - \left(\dfrac{(C+1)^2}{C}\right)D_1 \right)$;
	\item $b_{2,0}=-\dfrac{(C-1)^2}{C}$;
	\item $b_{1,1}=\dfrac{C^2-1}{C}$.
\end{itemize}

For $B_1=D_1=0$, we have  $b_{2,0}\neq 0$ and $a_{2,0}+b_{1,1} \neq 0$ where $a_{2,0}=0$ is the coefficient of $x^2$ in the first equation.
According to  Theorem 8.4 on \cite{KZ}, system is writen in the normal form of the Bogdanov-Takens family.

\end{proof}

\section{Examples} \label{s6}
In this section we  present some examples to illustrate our main results.\\

\begin{example}\end{example}
Let
\begin{equation}
X(x,y,z)= -\sgn(x)\frac{\partial}{\partial x}-\sgn(y)\frac{\partial}{\partial y}+\frac{\partial}{\partial z} \label{example1}
\end{equation}
be $2$-cross piecewise-smooth vector field defined on $\R^3$. Given a transition function $\varphi$, consider the regularization

\begin{equation}
X_{\varepsilon, \eta}=\frac{1}{4}\sum_{s \in \{-,+\}^2} \left(1 \pm \varphi(x/\varepsilon) \right)
\left(1 \pm \varphi(y/\eta)\right) X_{s}.\label{eqreg1}
\end{equation}

\begin{figure}[!ht]
\centering
\includegraphics[scale=0.35]{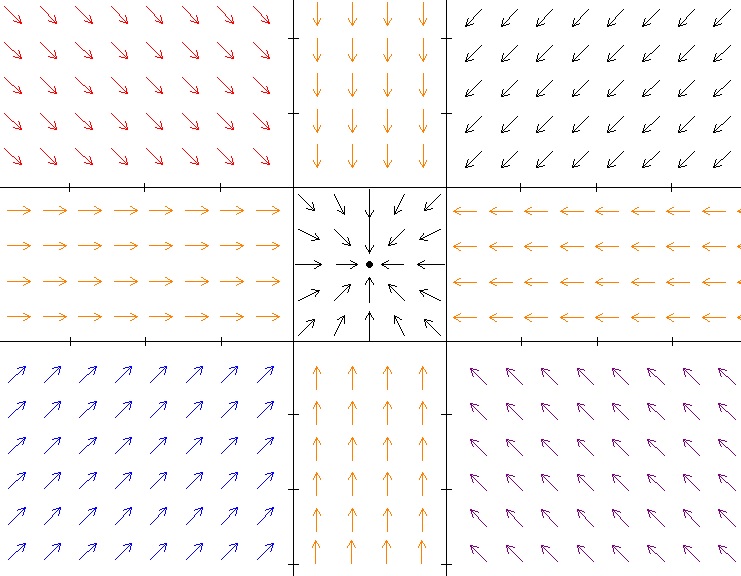}
\caption{\small Directional blow-up on the plane $z=z_0$ with vector fields $X_{++}$ in black, $X_{-+}$ 
	in red, $X_{--}$ in blue, $X_{+-}$ in purple, the vector field in orange is the convex combination 
	between $X_s$ and on the square is the $xy$-system.}
\end{figure}

We remark  that although the figures are planar, they are indicating that sliding occurs
along the entire $x_3$--axis. In fact, the third equation is $\dot{z}=1$.

Take the directional blow-up  $x\rightarrow \e x$,
$y\rightarrow \eta y$:

\begin{equation}
x'= -\frac{\varphi(x)}{\varepsilon};\quad y'= -\frac{\varphi(y)}{\eta};\quad z'= 1. \label {eqex11}
\end{equation}
For sufficiently small values of $\varepsilon,\eta$,  system (\ref{eqex11}) is topologically equivalent to
\begin{equation}
x'= -\frac{x}{\varepsilon};\quad y'= -\frac{y}{\eta};\quad z'= 1. \label{eqex12}
\end{equation}
The third equation does not affect the dynamics of \eqref{eqex12}. It's easy to see that, the only 
critical point of $xy$--system is the origin and it is asymptotically stable, for all $\e$ and $\eta$. 
So any  $p\in\Sigma_{00}$ is a sliding point and all trajectories of \eqref{eqex11} are attracted to $\Sigma_{00}$.

\begin{example}\end{example}
Let $X$ be the two parameters family of $2$-cross piecewise-smooth vector field defined on $\R^3$, given by 

\begin{eqnarray*}\small
\begin{split}
X_{++}=&\left(\frac{5}{36}-\frac{\alpha}{4},-\frac{1}{18}-\frac{\beta}{4},1\right) \ \ &X_{+-}=&\left(-\frac{13}{36}-\frac{\alpha}{4},-\frac{1}{18}-\frac{\beta}{4},1\right)\\ 
X_{--}=&\left(\frac{5}{36}-\frac{\alpha}{4},\frac{35}{18}-\frac{\beta}{4},1\right) \ \ &X_{-+}=&\left(-\frac{13}{36}-\frac{\alpha}{4},-\frac{1}{18}-\frac{\beta}{4},1\right),
\end{split}
\end{eqnarray*}
with $\alpha,\beta \in \R$.

Remember that $M_{\pm  \pm }$ are the regions where $X_{\pm  \pm }$ are defined.
Take the directional blow-up  $x\rightarrow \e x$,
$y\rightarrow \eta y$:
\begin{eqnarray}
\begin{split}
\e x'&= -\frac{4}{9}+\varphi(x)\varphi(y)-\alpha;\\ 
\eta y'&= -2\varphi(x)-2\varphi(y)+\frac{16}{9}-\beta+2\varphi(x)\varphi(y).\\
z'&= 1. \label{eqex21}
\end{split}
\end{eqnarray}
For sufficiently small values of $\varepsilon,\eta$,  system (\ref{eqex21}) is topologically equivalent to
\begin{eqnarray}
\e x'= -\frac{4}{9}+xy-\alpha;\quad \eta y'= -2x-2y+\frac{16}{9}-\beta +2xy.\label{eqex22}
\end{eqnarray}

First of all, consider $\e=\eta=1$ in order to analyze the bifurcation that appears when $\alpha=\beta=0$.
Now consider the following sequence of change of coordinates:
\begin{itemize}
	\item $x\rightarrow x+\dfrac{2}{3}$; $y\rightarrow y+\dfrac{2}{3}$;
	\item $t\rightarrow \dfrac{3}{2}t$;
	\item $x\rightarrow x+\dfrac{y}{3}$, $y\rightarrow \dfrac{2}{3}y-x$;
	\item $y\rightarrow y+\alpha-\dfrac{\beta}{2}$, $x\rightarrow x+\dfrac{1}{6}\alpha-\dfrac{\beta}{12}$.
\end{itemize}
We find

\begin{eqnarray}
x'= y,\quad y'= \mu +\nu y -\frac{9}{2}x^2+\frac{3}{2}xy+y^2,\label{eqex23}
\end{eqnarray}
where 
$$
\mu = -\frac{3}{2}\beta+\frac{9}{8}\alpha^2-\frac{9}{8}\alpha\beta+\frac{9}{32}\beta^2-\frac{3}{2}\alpha \text{\ and \ } \nu=\frac{9}{4}\left(\alpha - \frac{\beta}{2} \right).
$$
Notice that system \eqref{eqex23} is in the normal form of Bogdanov-Takens codimension $2$ bifurcation. Calculating the equilibrium points of \eqref{eqex22}, we get that for $\alpha$ and $\beta$ 
such that $36\alpha^2-36\alpha \beta+9\beta^2-48\alpha-48\beta>0$, the system have equilibrium points
(except when the both parameters are $0$, case with  only one equilibrium). Using the local expressions
 for the curves $H$ and $C$ we obtain the bifurcation diagram shown in Figure \ref{btdiag}.
\begin{figure}[!ht] 
\centering 
\includegraphics[scale=0.8]{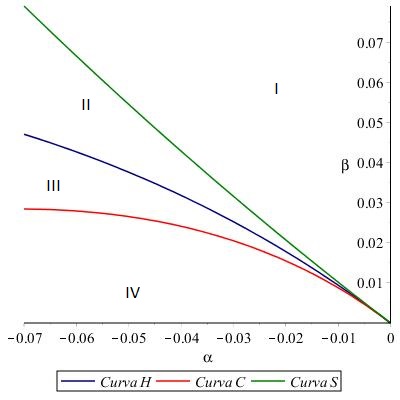}
\caption{\small Diagram bifurcation of \eqref{eqex23}.}
\label{btdiag}
\end{figure}

For the parameters in  region I there are no singularities. The curve $S$ is a curve
of generic saddle-node bifurcations, and thus there are a saddle and a repelling
equilibrium in II. From II to III we pass by the curve $H$, which denotes
a line of generic Hopf bifurcations. Consequently in III there are a saddle equilibrium
point, an attracting equilibrium point and a repelling limit cycle around the latter. The
limit cycle disappears in a (global) saddle loop bifurcation as we pass from III to
IV by the curve $C$ . Finally the attracting and the saddle equilibrium points in IV collapse
in a saddle node bifurcation as we pass back to I via S.

Consider again the parameters $\e$ and $\eta$ for the $xy$-system. 
As seen before, we have three possibilities for interactions between the parameters,
$\e/\eta = 0$, $\e/\eta = \infty$ or $\e/\eta = k$, $k>0$. For the two first interactions, 
we don't have the Bogdanov-Takens codimension $2$ bifurcation, thus the sliding will 
be decided using the Theorem \ref{mainprop}. But if we have the interaction $\e/\eta = 1$, 
the bifurcation is well defined and the phase portrait for each of the parameters in 
the regions described by Figure \ref{btdiag} remains unchanged when $\e, \eta\rightarrow 0$. 
Then, we check for which parameters there are a sliding region (positive time):
\begin{itemize}
\item If $\alpha, \beta\in I$, the $xy$-system does not have equilibria points, which implies that $X$ has no sliding region.

\item If $\alpha,\beta\in S$, the $xy$-system has a saddle-node equilibrium, attracting in the node side. 
$\Sigma_{+0}$ has only sewing region and any trajectory starting on $M_{++}$ crosses $\Sigma_{+0}$
and it is attracted to $\Sigma_{0-}$ or directly to $\Sigma_{00}$. If the trajectory intercts $\Sigma_{0-}$, it 
slides to  $\Sigma_{00}$. For a trajectory starting in $M_{-+}$, it is attracted to $\Sigma_{-0}$ and after it slides 
moving away from  $\Sigma_{00}.$  For a trajectory starting in $M_{--}$ three possibilities can occur: or it is attracted to $\Sigma_{-0}$
or to $\Sigma_{00}$ or  to $\Sigma_{0-}.$

\item If $\alpha, \beta\in II$,  $xy$-system has a saddle point and one attractive node; if
 $\alpha, \beta\in H\cup III$ $xy$-system has 
a saddle point and a focus (Hopf bifurcation), with an attracting limit cycle on the region III. 
For all cases $\Sigma_{10}$ has only sewing region and any trajectory starting on $M_{++}$ crosses
 $\Sigma_{+0}$ and it is attracted to $\Sigma_{0-}$ or directly to $\Sigma_{00}$. If the trajectory 
 intercts $\Sigma_{0-}$, it slides to  $\Sigma_{00}$. For a trajectory starting in $M_{-+}$, it is 
 attracted to $\Sigma_{-0}$ and after it slides 
moving away from  $\Sigma_{00}.$  For a trajectory starting in $M_{--}$ three possibilities can occur: 
or it is attracted to $\Sigma_{-0}$
or to $\Sigma_{00}$ or  to $\Sigma_{0-}.$ 

\item If $\alpha, \beta \in C\cup IV$  or $\alpha = \beta =0$, the trajectories of $X$ are as described in the previous case.
\end{itemize}

\begin{example}
\end{example}
Let $X$ be the $2$-cross piecewise-smooth vector field defined on $\R^3$, given by 
$X(x,y,z)=X_{\pm  \pm }(x,y,z)$ where
\begin{eqnarray*}\small
	\begin{split}
		X_{++}=&\left(\frac{259}{1800},\frac{13969}{351900}-\frac{\sqrt{13519}}{1173},1\right) \ \ &X_{+-}=&\left(-\frac{641}{1800},\frac{13969}{351900}-\frac{\sqrt{13519}}{1173},1\right)\\ 
		X_{--}=&\left(\frac{259}{1800},\frac{717769}{351900}-\frac{\sqrt{13519}}{1173},1\right) \ \ &X_{-+}=&\left(-\frac{641}{1800},\frac{59}{900}-\frac{\sqrt{13519}}{1173},1\right),
	\end{split}
\end{eqnarray*}
Give a transition function $\varphi$ consider   $t\rightarrow \e t$, $x\rightarrow \e x$ and $y\rightarrow \eta y.$ We get
\begin{eqnarray}
\begin{split}
x'&= xy-\frac{191}{450};\\ 
y'&= \frac{\e}{\eta}\left(\frac{189919}{87975}-\frac{4\sqrt{13519}}{1173}-2\left(x+y+xy\right) \right);\\
z'&= 1, 
\end{split}
\end{eqnarray}
We discuss the influence of parameters on the stability of the equilibrium points, 
which define the slide for the $2$-cross piecewise-smooth vector field $X$.There exist exactly two 
equilibrium and we already know that one is a saddle. 
Consider the non-saddle equilibrium $P=(P_1,P_2)$, where
$$\small
P_1=-\frac{74681}{15\left(21\sqrt{2}\sqrt{13519}-2100\sqrt{2}+10\sqrt{13519}-8820\right)}
$$
and
$$\small
P_2= -\frac{\sqrt{13519}}{1173} +\frac{294}{391}-\frac{7\sqrt{27038}}{3910} +\frac{70\sqrt{2}}{391} .
$$
Consider the linearization of the $xy$-system at $P$. 
\begin{itemize}
	\item[i.] For $\e/\eta \rightarrow 1$ when $\e, \eta\rightarrow 0$,  the eigenvalues $\lambda_1$ and 
	$\lambda_2$ are both negative, which implies that $P$ is an attractive node.
	\item[ii.] For $\e/\eta \rightarrow \infty$ when $\e, \eta\rightarrow 0$, the eigenvalues $\lambda_1$ 
	and $\lambda_2$ are both negative, which implies that $P$ is an attractive node.
	\begin{figure}[!htb]
		\begin{center}
			\includegraphics[height=3cm]{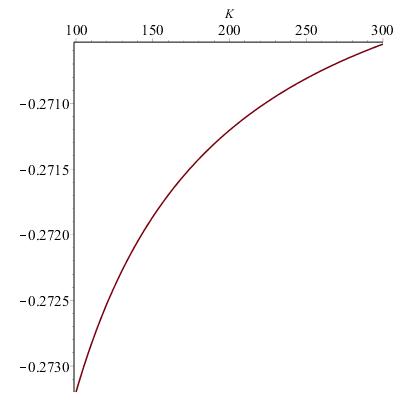} \quad
			\includegraphics[height=3cm]{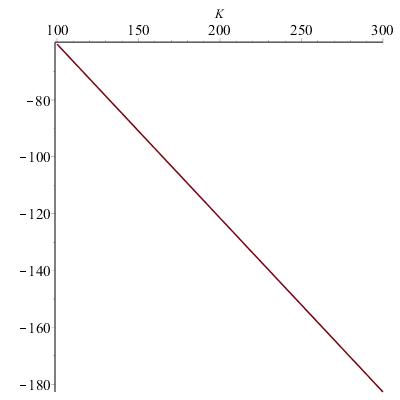}
			\caption{Values of the eigenvalues as $K=\e/\eta$ goes to $\infty$.}
		\end{center}
	\end{figure}
	
	\item[iii.] For $\e/\eta \rightarrow 0$ when $\e, \eta\rightarrow 0$  the eigenvalues $\lambda_1$ and 
	$\lambda_2$ are both positive, which implies that $P$ is a repelling node.
	\begin{figure}[!htb]
		\begin{center}
			\includegraphics[height=3cm]{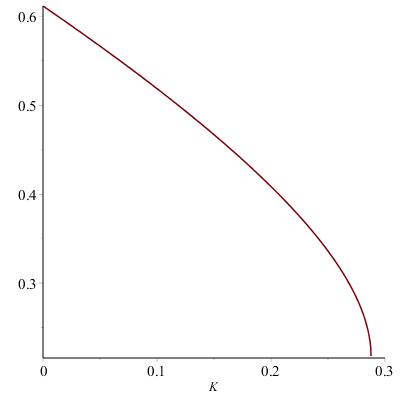} \quad
			\includegraphics[height=3cm]{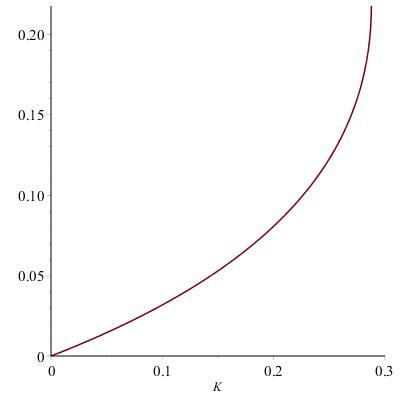}
			\caption{Values of the eigenvalues when $K=\e/\eta\rightarrow 0$.}
		\end{center}
	\end{figure}
	Observe that, this results are consisting with Proposition 4.2. Thus, the origin is a sliding point.
\end{itemize}

\begin{example}\end{example}

Let $X$ be the $2$-cross piecewise-smooth vector field defined on $\R^3$, given by 

\begin{eqnarray*}
\begin{split}
X_{++}=&\left(\frac{277}{1800},-\frac{59}{900},1\right), \ \ &X_{+-}=&\left(-\frac{623}{1800},-\frac{59}{900},1\right),\\ 
X_{--}=&\left(\frac{277}{1800},\frac{1741}{900},1\right), \ \ &X_{-+}=&\left(-\frac{623}{1800},-\frac{59}{900},1\right).
\end{split}
\end{eqnarray*}
\begin{figure}[!ht]
\centering
\includegraphics[scale=0.4]{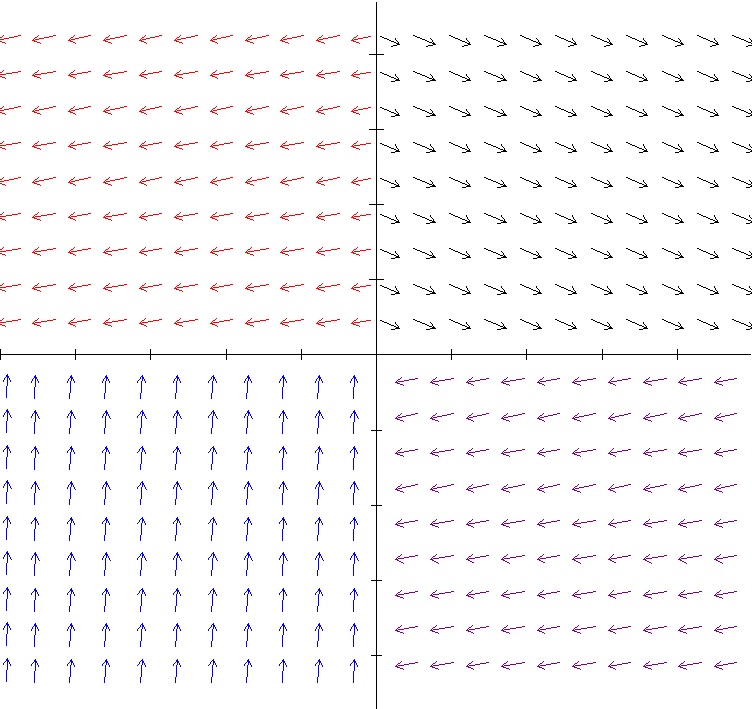}
\caption{\small Sketch of the phase portrait of $X$ on the plane $z=z_0$ with vector fields $X_{++}$ 
	in black, $X_{-+}$ in red, $X_{--}$ in blue and $X_{+-}$ in purple.}
\end{figure}
$X$ is obtained  from the previous example, with $\alpha=-6/100$ and $\beta=4/100$ (region III). 
Consider the initial condition $P_1=(0.5,0.5).$ 
The trajectory of $X_{++}$ starting in $P_1$ is
\begin{equation}
x(t) = (277/1800)t+1/2, y(t) = -(59/900)t+1/2.
\end{equation}
 After $t=450/59$, the trajectory reaches $\Sigma_{+0}$ at $P_2=(395/236,0)$ and as seen before, 
 $\Sigma_{+0}$ is a sewing region. 
 So we take the  trajectory of $X_{+-}$  by $P_2$ when $t=450/59$ 
\begin{equation}
x(t) = -(623/1800)t+509/118, y(t) = -(59/900)t+1/2.
\end{equation}

After $t=458100/36757$ the solution reaches $\Sigma_{0-}$, which is a sliding manifold. Calculating 
the Filippov sliding vector field on $\Sigma_{0-}$,
$$
X_{0-}^{sl}= \left[ \begin {array}{c} 0\\ \noalign{\medskip}{\dfrac {1187}{900}}
\end {array} \right] 
$$
and its trajectory by  $P_3=(0,-395/1246)$ at $t=458100/36757$, we get
\begin{equation}
x(t) = 0, y(t) = (1187/900)t-1977/118.
\end{equation}
Finally, after $t=\frac{1779300}{140066}$ the solution reaches and slides on $\Sigma_{00}$ as expected.
\begin{figure}[!ht] 
\centering 
\includegraphics[scale=0.4]{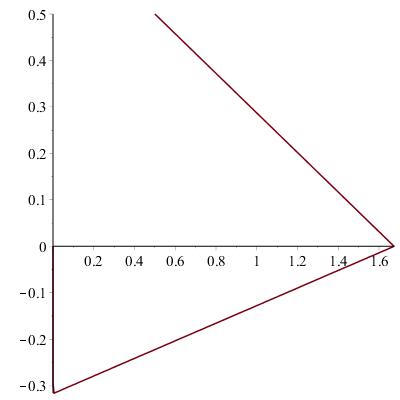}
\caption{\small Plot of the piecewise solution starting on $(0.5,0.5)$.}
\end{figure}

The straightline 
$x \left( t \right) =\frac {277t}{1800}-\frac {277}{3482}$, $y \left( t \right) =\frac {1741t}{900}-1$
located in the third quadrant, Figure \ref{hachu}, delimits the attraction basin of  $\Sigma_{00}$ 
(hatched region) and of  $\Sigma_{-0}$ (no hatched region).

\begin{figure}[!htb]
	\begin{center}
		\includegraphics[height=6cm]{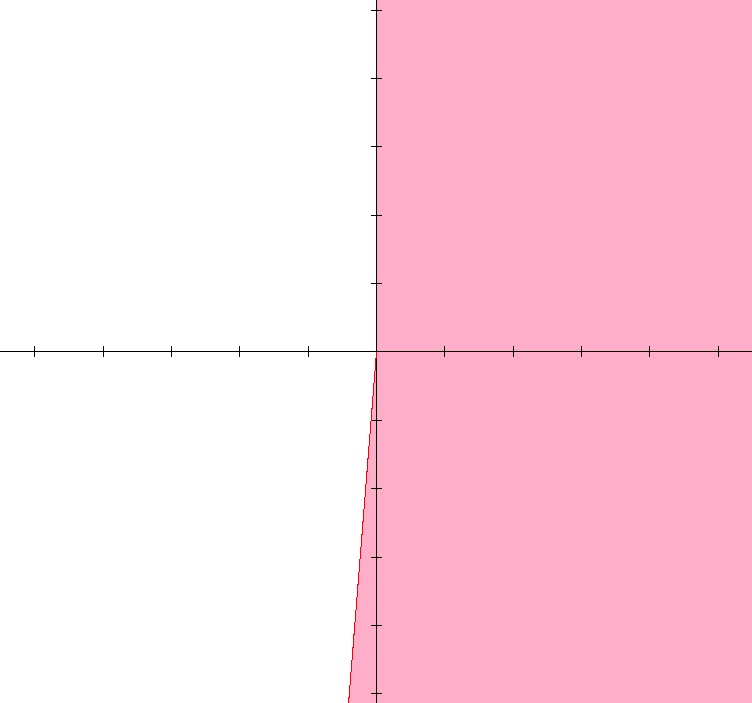}
		\end{center}\caption{\small Attraction basin of $\Sigma_{00}$ (hatched region) and of 
		 $\Sigma_{-0}$ (no hatched region).}
\label{hachu}\end{figure}

\section{Acknowledgments} The authors are partially supported by CAPES and FAPESP.
The authors are grateful for the suggestions and comments of Daniel Cantergiani Panazzolo and for the hospitality of LMIA-UNIVERSITÉ DE HAUTE-ALSACE.

\end{document}